\documentclass[10pt]{amsart}
\usepackage{tikz-cd,enumerate,ulem}
\usepackage{mathtools}
\usetikzlibrary{arrows}

\usepackage[utf8]{inputenc}
\usepackage{xspace}

\usepackage[english]{babel}
\usepackage[T1]{fontenc}
\usepackage{amsmath,amssymb}

\usepackage[a4paper,top=3cm,bottom=2cm,left=3cm,right=3cm,marginparwidth=1.75cm]{geometry}

\usepackage{amsthm}
\usepackage{amsfonts}
\usepackage{graphicx}
\usepackage[colorinlistoftodos]{todonotes}
\usepackage[colorlinks=true, allcolors=blue]{hyperref}
\usepackage{thmtools}
\usepackage{cleveref}
\usepackage{url}
\newtheorem{theorem}{Theorem}[section]
\newtheorem{thm}[theorem]{Theorem}

\newtheorem{proposition}[theorem]{Proposition}

\newtheorem{lemma}[theorem]{Lemma}

\newtheorem{example}[theorem]{Example}

\newtheorem{lmm}[theorem]{Lemma}

\newtheorem{prop}[theorem]{Proposition}
\newtheorem{cor}[theorem]{Corollary}

\theoremstyle{definition}

\newtheorem{rem}[theorem]{Remark}
\newtheorem{remark}[theorem]{Remark}

{
      \theoremstyle{plain}
      
  }
  {
      \theoremstyle{plain}
      
  }

\DeclareMathOperator{\FPdim}{FPdim}

\usepackage{enumitem}
\setlist[itemize,1]{label={--\,}}
\setlist{leftmargin=8mm}

\newtheorem*{fact*}{Fact} 
\newcommand{\Z}{{\mathbb{Z}}}

\newcommand{\Q}{{\mathbb{Q}}}

\newcommand{\ccirc}{\kern0.5ex\vcenter{\hbox{$\scriptstyle\circ$}}\kern0.5ex}

\newtheorem*{theorem*}{Theorem}
\newtheorem*{remark*}{Remark}

\numberwithin{equation}{section}

\newcommand{\poubelle}[1]{}

\title[Vanishing fusion rings of rank four]{Classification of vanishing fusion rings of rank four}
\author{Jingcheng Dong}
\address{College of Mathematics and Statistics, and Center for Applied Mathematics of Jiangsu Province, Nanjing University of Information Science and Technology, Nanjing 210044, China}
\email{jcdong@nuist.edu.cn}
\author{S\'ebastien Palcoux}
\address{Beijing Institute of Mathematical Sciences and Applications, Huairou District, Beijing 101408, China}
\email{sebastienpalcoux@gmail.com}
\author{Arnaud Plessis}
\address{Beijing Institute of Mathematical Sciences and Applications, Huairou District, Beijing 101408, China}
\email{plessis@bimsa.cn}
\date{}

\begin{document}
\begin{abstract}
A fusion ring is said to be vanishing if one of its multiplication matrices has determinant zero. 
We classify all vanishing fusion rings of rank four. 
We then provide a more explicit classification in the integral case, which leads to new consequences: there are infinitely many nonisomorphic simple integral vanishing fusion rings of rank four and no simple integral fusion ring of rank four is $1$-Frobenius. We also classify all Burnside fusion rings of rank four.  
\end{abstract}
\maketitle

\tableofcontents

\section{Introduction}
We start this paper by introducing the terminology that we will use throughout the paper. 
Let $\mathcal R$ be a free $\mathbb Z$-module with basis
$\{X_1,\ldots,X_r\}$ and multiplication
\[ X_iX_j=\sum_{k=1}^r N_{ijk}X_k, \] where $N_{ijk}\in\mathbb Z_{\geq0}$ are the \textit{structure constants} of the multiplication. 
Given an index $i$, the \textit{multiplication matrix} of $X_i$ is $N_{X_i}=(N_{ijk})_{1\leq j,k\leq r}$, where 
$j$ indexes rows and $k$ indexes columns.

The \emph{basic Frobenius--Perron dimension} $\FPdim(X_i)$ of $X_i$ is the spectral
radius of $N_{X_i}$ and the \textit{global Frobenius--Perron dimension} $\FPdim(\mathcal{R})$ of $\mathcal{R}$ is the sum of all basic Frobenius--Perron dimensions squared.  
We say that $\mathcal{R}$ is \textit{integral} if $\FPdim(X_i)$ is an integer for all $i$.

We now introduce the following four axioms:
\begin{itemize}
\item (Associativity) $\forall i,j,k$, one has $X_i(X_jX_k)=(X_iX_j)X_k$; 
\item (Unit) $\forall i$, one has $X_1X_i=X_iX_1=X_i$; 
\item (Dual) $\forall i, \exists! i^*\in\{1,\dots,r\}, \forall j$ one has $N_{ij1}=N_{ji1}=\delta_{i^*j}$, where $\delta_{lm}$ is the Kronecker symbol; 
\item (Frobenius reciprocity) $\forall i,j,k$, one has $N_{ijk}=N_{i^*kj}=N_{kj^*i}$. 
\end{itemize}

The free module $\mathcal{R}$ is called \emph{self-dual} if $i^*=i$ for every $i$ and is called a \textit{fusion ring} if the four axioms above hold ($N_{X_1}$ is then the identity matrix and $N_{X_{i^*}}=N_{X_i}^T$ for all $i$, where $A^T$ denotes the transpose of a matrix $A$). 
In the latter case, a fusion ring $\mathcal{R}'$ with basis $\{X'_1,\dots, X'_r\}$ and structure constants $N'_{ijk}$ is \textit{isomorphic} to $\mathcal{R}$ if there is a permutation $\sigma$ of $\{1,\dots,r\}$ such that $N'_{ijk}=N_{\sigma(i)\sigma(j)\sigma(k)}$ for all $i,j,k\in\{1,\dots,r\}$. 

A fusion ring is said to be \textit{simple} if it has no proper nontrivial fusion subring. 

The new condition studied here is the singularity of a multiplication
matrix. 
Letting $\det(A)$ be the determinant of a square matrix $A$, we say that $\mathcal{R}$ is \textit{vanishing} if $\det(N_{X_i})=0$ for some $i$. 

\subsection{Motivation for studying the vanishing property} \label{ss-sec:motivation}
For a commutative fusion ring, the eigenvalues of $N_X$, where $X$ is any basis element, are the values of its ring characters; thus vanishing is equivalent to the occurrence of a zero in the
character table. This gives the condition both a spectral meaning and
a character-theoretic interpretation. 

A first motivation comes from Burnside's vanishing theorem:
every nonlinear irreducible character of a finite group vanishes on
some conjugacy class \cite[Theorem 3.15]{Isaacs}. Consequently, the Grothendieck ring of
$\operatorname{Rep}(G)$ is vanishing for every finite nonabelian group $G$, where $\operatorname{Rep}(G)$ denotes the category of finite-dimensional complex representations of $G$.
Burnside's theorem is substantially stronger than the existence of a
single zero: it requires a zero for every noninvertible irreducible
representation. Burciu and the second author extend this perspective in \cite[Subsection 12.1]{BP}: a fusion ring with unit $\mathbf{1}$ is said to be \textit{Burnside} if $\det(N_X)=0$ for any basis element $X$ satisfying $XX^*\neq \mathbf{1}$. 
Every nonpointed Burnside fusion ring is therefore vanishing.
The weaker vanishing condition is a natural object of study in its own
right: it isolates the presence of a zero without requiring the
vanishing in every such row supplied by Burnside's theorem.

The vanishing condition occurs frequently among Grothendieck rings of small
rank. Of the fifteen based rings of rank four admitting a pivotal categorification
over $\mathbb C$ classified by the authors
\cite[Theorem~1.1]{DPP}, exactly seven are vanishing.
They are listed below; the second column gives the Frobenius--Perron
dimension of one simple object whose multiplication matrix has determinant zero. 
The count follows directly from the fusion rules in that classification.
Here, $D_{10}$ denotes the dihedral group of order $10$ and
$\operatorname{NG}(C_3,m)$ denotes a near-group category with
noninvertible simple $X$ and fusion rules $gX=Xg=X$ and
$X^2=\sum_{g\in C_3}g+mX$.

\begin{center}
\small
\renewcommand{\arraystretch}{1.2}
\begin{tabular}{lc}
\hline
Categorical model & $\FPdim(X)$ with $\det(N_X)=0$ \\
\hline
$\operatorname{Rep}(D_{10})$ & $2$ \\
$\operatorname{Ad}(\operatorname{SU}(2)_7)$
  & $1+2\cos(2\pi/9)$ \\
Tambara--Yamagami on $C_3$ & $\sqrt{3}$ \\
Haagerup $\mathcal H_1$ & $(1+\sqrt{13})/2$ \\
$\operatorname{Rep}(A_4)$ & $3$ \\
$\operatorname{NG}(C_3,3)$ (Izumi--Xu) & $(3+\sqrt{21})/2$ \\
$\operatorname{NG}(C_3,6)$ & $3+2\sqrt{3}$ \\
\hline
\end{tabular}
\end{center}

Likewise, fourteen of the seventeen presently known rank-five
Grothendieck rings are vanishing; the known examples are supplied by
\cite{LPR,ABDP,BNRW,Izumi}. This is a count among known based rings,
not an assertion of a complete classification in rank five.
Different categorifications of the same fusion ring are counted only
once. These small-rank examples suggest that the vanishing condition captures a
substantial part of the arithmetic of categorifiable rings. They also
motivate an asymptotic question: for integers $r,B\geq 1$, let $\mathcal{G}_{r,B}$ be the set of based-ring isomorphism classes of Grothendieck rings of rank $r$ admitting a pivotal categorification over $\mathbb{C}$ and having all structure constants at most $B$. 
Let $\mathcal{V}_{r,B}$ be its subset of vanishing rings, and put $\delta_r$ to be the limit inferior of $\vert \mathcal{V}_{r,B} \vert / \vert \mathcal{G}_{r,B} \vert$ as $B\to +\infty$, where $\vert S\vert$ is the size of a finite set $S$. 
Does $\delta_r$ tend to $1$ as $r$ grows?

We now provide one final motivation, phrased as a question: Can a fusion ring admit a categorification in positive characteristic but none in characteristic zero \cite[Question~5.3]{LPR}? To our knowledge, no such example is known. 
Such an integral fusion ring must have rank at least four: every integral fusion ring of rank at most three is isomorphic to $\mathbb Z$, $\mathbb Z[C_2]$, $\mathbb Z[C_3]$, or $K_0(\operatorname{Rep}(S_3))$ \cite[proof of Theorem~5.1]{Cze}, and all of them admit a characteristic-zero categorification.

A consequence of \cite[Theorem 1.1]{ABDP} is that no simple integral fusion ring of rank four admits a categorification in characteristic zero. 
Therefore, obtaining such explicit fusion rings is of interest, as they are candidates for the first examples admitting a categorification in positive characteristic but none in characteristic zero. 

We construct in Theorem \ref{thm:classintegral1} an explicit infinite family of simple integral vanishing fusion rings. 
If any member of this family admits a categorification $\mathcal{C}$ over an algebraically closed field of positive characteristic, then the categorical global dimension of $\mathcal{C}$, defined as the sum of the squared norms of the simple objects, must be $0$ \cite[Theorem~9.3]{ENO05}, see also \cite[Theorem~9.16.1]{EGNO}. 
 \subsection{Main results}
The first aim of this paper is to provide a complete parametrization, possibly with repetitions, of the isomorphism classes of vanishing fusion rings of rank four.
Let us fix once and for all a free $\Z$-module $\mathcal{R}$ with basis $\{X_1, X_2, X_3, X_4\}$ and multiplication
\[ X_iX_j=\sum_{k=1}^4 N_{ijk}X_k, \qquad N_{ijk}\in\mathbb Z_{\geq0},\] satisfying the unit, duality, and Frobenius reciprocity axioms above. 
As fusion rings of rank four are commutative \cite[Proposition 9.2]{ABDP}, we also suppose that the multiplication is commutative. 
Since we are interested in isomorphism classes of vanishing fusion rings, we may assume without loss of generality that $\det(N_{X_2})=0$. 
For brevity, put $\mathbf{1}=X_1, X=X_2, Y=X_3$, and $Z=X_4$. 

Thus, only associativity remains to be checked. This is equivalent to showing that $N_X, N_Y$, and $N_Z$ commute pairwise \cite[Definition~1.2 and Remark~1.5]{BDD}, see also \cite[Lemma A.1]{DPP} for a result more in the spirit of this paper (the matrix of $L_{X_i}$ in the basis $\{X_1,X_2,X_3,X_4\}$ is $N_{X_i}^T$).

We now explain the content of this text.
 Section \ref{sec:nonselfdual} treats the
non-self-dual case. In the self-dual case, the tensor $(N_{ijk})$ is
completely symmetric and the multiplication matrices are of the form
\begin{equation}\label{eq:matrices}
\begin{gathered}
 N_X=\begin{pmatrix}
 0&1&0&0\\1&a&b&c\\0&b&d&e\\0&c&e&f
 \end{pmatrix},\qquad
 N_Y=\begin{pmatrix}
 0&0&1&0\\0&b&d&e\\1&d&g&h\\0&e&h&i
 \end{pmatrix},\; \; \; \; \text{and} \; \; \; \; 
 N_Z=\begin{pmatrix}
 0&0&0&1\\0&c&e&f\\0&e&h&i\\1&f&i&j
 \end{pmatrix},
\end{gathered}
\end{equation}
where the entries are nonnegative integers. 
Section \ref{sec:selfdualdeg} settles the degenerate case $df=0$. 
In these preliminary sections, none of the fusion rings is simple and there are exactly four isomorphism classes of integral vanishing fusion rings of rank four, one of which is non-self-dual (Remark \ref{rem:isod=0f>0}).

Let $\gcd(x,y)$ denote the greatest common divisor of positive integers $x$ and $y$. 
The nondegenerate case $df>0$ is described by the next theorem, whose proof is given in Subsections \ref{subsec:reductions} and \ref{subsec:linearmatrix}. 
Set
\begin{equation} \label{eq:kappapq}
 \kappa=b^2+c^2+(a-f-d)\gcd(d,f),\; \;  \; \; \; \; p=\sqrt{\frac{f}{\gcd(d,f)}},\; \; \; \; \; \; \text{and} \; \;  \; \; \; \; q=\sqrt{\frac{d}{\gcd(d,f)}} .
\end{equation}

We will show at the very beginning of Subsection \ref{subsec:reductions} that $p$ and $q$ are coprime positive integers.
Swapping $Y$ and $Z$ sends the tuple $(a,b,c,d,e,f,g,h,i,j)$ to $(a,c,b,f,e,d,j,i,h,g)$. In particular, this changes the sign of $qc-pb$ and we can therefore always assume that it is a nonnegative integer. 

\begin{theorem}\label{thm:main}
Assume that $\mathcal{R}$ is self-dual and that $df>0$.
Then the following are equivalent:
\begin{enumerate}
\item $\mathcal R$ is a fusion ring.
\item The ten structure constants satisfy:
 \begin{enumerate}
 \item $qc-pb=1$;
 \item $bq+pc$ is a positive integer dividing $c^2+b^2+(p^2+q^2)\kappa$;
 \item
 \[
  \begin{pmatrix}g&h\\h&i\end{pmatrix}=cR-pK \; \; \text{and} \; \; 
  \begin{pmatrix}h&i\\i&j\end{pmatrix}=-bR+qK,
 \]
 where \[ \tau =\frac{c^2+b^2+(p^2+q^2)\kappa}{bq+pc}-(bq+pc), \qquad R =\begin{pmatrix}
   2qb-q^2\tau&qc+pb-pq\tau \\
   qc+pb-pq\tau&2cp-p^2\tau
   \end{pmatrix},\] and  \[ K=\begin{pmatrix}
   b^2-1+q^2(c^2+b^2-\kappa)&bc+pq(c^2+b^2-\kappa)\\
   bc+pq(c^2+b^2-\kappa)&c^2-1+p^2(c^2+b^2-\kappa)
   \end{pmatrix}.\]
 \end{enumerate}
\item $\mathcal R$ is a simple fusion ring.
\end{enumerate}
\end{theorem}

\begin{remark}\label{rem:divisibility}
Part $(2b)$ of Theorem~\ref{thm:main} is redundant, since it is implied by $(2a)$ and $(2c)$, but it highlights that $\tau$ is an integer.  Indeed, $bq+pc$ cannot be $0$ since either $b$ or $c$ is positive by $(2a)$. 
Next, $(2a)$ and $(2c)$ give
\[
 R=q\begin{pmatrix}g&h\\h&i\end{pmatrix}
   +p\begin{pmatrix}h&i\\i&j\end{pmatrix}.
\]
Thus, $R$ has integer entries, and its diagonal entries show that
$q^2\tau$ and $p^2\tau$ are integers. Since $p$ and $q$ are coprime,
B\'ezout's identity gives $\alpha,\beta\in\Z$ such that $\alpha p^2 +\beta q^2 =1$, and so $\tau=\alpha p^2 \tau + \beta q^2\tau \in\mathbb Z$.
\end{remark}

\begin{remark}
The overlap of entries $h$ and $i$ in Theorem \ref{thm:main} $(2c)$ is not an additional restriction: the relation $qc-pb=1$, which implies $(q^2+p^2)(b^2+c^2)-(bq+cp)^2=(qc-pb)^2=1$ by Lagrange identity, forces the $(2,1)$ and $(2,2)$-entries of the symmetric matrix $cR-pK$ to match $(1,1)$ and $(1,2)$-entries of the symmetric matrix $-bR+qK$, respectively. 
\end{remark}

We now illustrate through a concrete example how this theorem applies. 

\begin{example}\label{ex:family}
\rm{Fix an integer $\kappa$ as well as 
\begin{equation} \label{eq:example}
\det(N_X)=0,\; \; \; \; b=2,\; \; \; \;  c=q=3,\; \; \text{and}\; \;  p=4.
\end{equation}
Then, $qc-pb=1$ and Theorem \ref{thm:main} $(2b)$ tells us that $18$ divides $13+25\kappa$, that is, $\kappa\equiv 11\pmod{18}$.
Next, Theorem \ref{thm:main} $(2c)$ gives
\[
 g=\frac{45-3\kappa}{2},\qquad h=25-2\kappa,\qquad
 i=\frac{112-8\kappa}{3},\; \; \text{and} \; \; j=\frac{424-32\kappa}{9}.
\]
Under the congruence $\kappa\equiv 11\pmod{18}$, these numbers are integers and they are all
nonnegative exactly when $\kappa\leq12$. Moreover,
\[
 d=9\gcd(d,f),\qquad e=12\gcd(d,f),\qquad f=16\gcd(d,f),
 \; \; \text{and} \; \; a=25\gcd(d,f)+\frac{\kappa-13}{\gcd(d,f)}.
\]

Consequently, the fusion rings satisfying $\kappa\leq 12, \kappa\equiv 11\pmod{18}$, and \eqref{eq:example} are in bijection with
the positive divisors $\lambda$ of $\kappa-13$ satisfying $25\lambda^2+\kappa-13\geq 0$.
Explicitly, $\lambda=\gcd(d,f)$ and \[(a,b,c,d,e,f,g,h,i,j) = \left(
 25\lambda+\frac{\kappa-13}{\lambda}, 2, 3, 9\lambda,
 12\lambda,16\lambda,
 \frac{45-3\kappa}{2}, 25-2\kappa,
 \frac{112-8\kappa}{3},\frac{424-32\kappa}{9}
 \right).\]
 Since the choice $\lambda=13-\kappa$ is admissible, there is at least one fusion ring satisfying the prescribed values above. 
Finally, there is no fusion ring satisfying \eqref{eq:example} if $\kappa\geq13$ or $\kappa\not\equiv 11 \pmod{18}$. }
\end{example}

Recall that an integral fusion ring of rank at most three is isomorphic to $\mathbb Z$, $\mathbb Z[C_2]$, $\mathbb Z[C_3]$, or $K_0(\operatorname{Rep}(S_3))$. 
What about the rank-four case? 
The existence of infinitely many simple integral fusion rings of rank four was raised
in \cite[slide~9]{Pal23}, \cite[slides~8 and~19]{Pal24}, and
\cite[Question~1]{Pal}. The same expectation is supported by the
computations of Bruns and the second author \cite{BPData}.
We answer this question affirmatively by constructing an explicit
infinite family of pairwise nonisomorphic simple integral vanishing
fusion rings of rank four.

However, we cannot immediately deduce anything about the integral case from Theorem \ref{thm:main}, which will be studied in Subsection \ref{ss-section:integral case}. 
We will show, among other things, that it corresponds precisely to the case $\kappa=\gcd(d,f)^2$ (Proposition \ref{prop:firstcharinteger}). 
In this proposition, we will also explicitly compute the basic Frobenius--Perron dimensions of $X,Y$, and $Z$.

The special case $\kappa=\gcd(d,f)^2$ is very rigid. 
For instance, the rich theory of Markov equations, only a tiny part of which is described in Subsection \ref{ss-sec:Markov},
forces the quotient of $c^2+b^2+(p^2+q^2)\kappa$ by $bq+pc$ to be $3\gcd(d,f)$ (Lemma \ref{lem:scale-onebis}). 
With this new equation at hand, it is then possible to provide a more explicit description of simple integral vanishing fusion rings of rank four. 
The proof is given in Subsection \ref{ss-sec:proofintegralcase}.
Set $(F_n)_{n\in \Z}$ to be the usual Fibonacci sequence, that is, $F_0=0, F_1=1$, and $F_{n+2}=F_{n+1}+F_n$ for all integers $n\in\Z$. 
In particular, $F_{-1}=1$ and $F_{-2}=-1$.

\begin{thm} \label{thm:classintegral}
Assume that $\mathcal{R}$ is self-dual and that $df>0$. 
Then $\mathcal{R}$ is a simple integral fusion ring if and only if there is an integer $n\geq 3$ such that \[\begin{cases} 
 \gcd(d,f)=\frac{F_{2n+1}}{p^2+q^2} \\ 
 
 \\
 
a=(1+p^2+q^2)\gcd(d,f)-\frac{b^2+c^2}{\gcd(d,f)} \\ 

\\

b=\frac{qF_{2n-1}-p}{p^2+q^2}, \; \; c=\frac{pF_{2n-1}+q}{p^2+q^2} \\

\\

d=\gcd(d,f)q^2, \; \; e=\gcd(d,f)pq, \; \; f=\gcd(d,f) p^2,  \\ 

\\

g= q^2\left(p\gcd(d,f)^2-3c\gcd(d,f)+b\right)+p+2b+pb^2 \\

\\

h = q\left(pb+c^2-3\gcd(d,f)pc+p^2\gcd(d,f)^2\right) \\

\\

i= p \left( qc+c^2-3pc \gcd(d,f) + p^2 \gcd(d,f)^2\right) \\ 

\\ 

j= (p^2-qc+2)c +p^2\gcd(d,f)\left( 3b - \gcd(d,f) q\right)-q
\end{cases}. \]
\end{thm}

Note that it suffices to choose $n,p$, and $q$ to compute all these constants. 
We cannot yet deduce from this theorem that there are infinitely many simple integral vanishing fusion rings of rank four; the constants defined in the displayed system can fail to be both integral and nonnegative.
For instance, the case $n=4, p=1$, and $q=4$ gives $b=3, c=1, \gcd(d,f)=2$, but $j=-3$.
On the other hand, the case $n=7, p=17$, and $q=4$ provides $\gcd(d,f)=2$ and \[(a,b,c,d,e,f,g,h,i,j)=(523, 3, 13, 32, 136, 578, 64, 200, 867, 3681). \]
 
Remarkably, if $\gcd(d,f)=1$, then for all integers $n\geq 3$, there are parameters $p$ and $q$ making all these constants both integral and nonnegative, see Subsection \ref{ss-sec:casegcd1} for a proof of this claim. 

\begin{thm} \label{thm:classintegral1}
Assume that $\mathcal{R}$ is self-dual and that $df>0$. 
If $\gcd(d,f)=1$, then $\mathcal{R}$ is a simple integral fusion ring if and only if there is an integer $n\geq 3$ such that \[ \begin{cases} 
a=1+F_{2n+1}-F_{2n-3} \\ 

\\ 

b=F_{k-2}, \; \; \; \; c=F_{l-2}, \; \; \; \; d=F_k^2, \; \; \; \; e=F_nF_{n+1}, \; \; \; \; f=F_l^2 \\

\\

g= F_k^2(F_{k-2}-F_{l-4})+F_l+2F_{k-2}+F_lF_{k-2}^2 \\ 

\\ 

h=F_k(F_{k-2}F_l-1), \; \; \; \; \; \; i=F_l(F_kF_{l-2}-1) \\ 

\\ 

j= F_{l-2}(F_l^2-F_{l-2}F_k+2)+F_l^2F_{k-4}-F_k
\end{cases}, 
\] where $l$ and $k$ are the odd and even integers belonging to $\{n,n+1\}$, respectively. 
In addition, all the constants defined in the displayed system are both integral and nonnegative. 
Finally, we have $(\FPdim(X), \FPdim(Y), \FPdim(Z))=(1+F_{2n+1}, \; F_{k-2}+F_kF_{2n-1}, \; F_{l-2}+F_lF_{2n-1})$.  
\end{thm}

In particular, $\FPdim(\mathcal{R})=1+ (1+F_{2n+1})^2+ (F_{n-2}+F_nF_{2n-1})^2+(F_{n-1}+F_{n+1}F_{2n-1})^2$. 
Since the Fibonacci sequence is strictly increasing with $n\geq 2$, this global dimension, which is invariant under fusion ring isomorphism, is strictly increasing with $n\geq 3$. Hence, there are infinitely many simple integral vanishing fusion rings of rank four. 
Due to its strong connection to the Fibonacci numbers, the infinite family stated in Theorem \ref{thm:classintegral1} is called the \textit{rank-four Fibonacci family}.

We now present two applications of our study, which will be proved in the final Section \ref{subsec:frobenius}.

A second question asks whether a simple integral fusion ring of rank
four can be $1$-Frobenius, meaning that every basic Frobenius--Perron
dimension divides the global one. This question, with nonexistence
suggested by the computations made in \cite{BPData}, appears in \cite[slide~9]{Pal23},
\cite[slide~8]{Pal24}, and \cite[Question~2]{Pal}. 
We prove that no such ring exists. 
More precisely:

\begin{thm} \label{thm:frobenius-obstruction}
In a simple integral fusion ring $\mathcal{F}$ of rank four, there is at most one nonunit basis element whose basic Frobenius--Perron dimension divides $\FPdim(\mathcal{F})$. 
Moreover, such an element exists if and only if, up to isomorphism, $\mathcal{F}$ belongs to the rank-four Fibonacci family.
\end{thm}

Any pointed fusion ring is (vacuously) Burnside. 
Our study completes the classification of Burnside fusion rings of rank four (note that a vanishing fusion ring is necessarily nonpointed). 

\begin{thm} \label{thm:Burnside}
A nonpointed fusion ring of rank four is Burnside if and only if it is non-simple vanishing. 
\end{thm}

In particular, no simple fusion ring of rank four is Burnside. 
Indeed, on the one hand, a nonpointed Burnside fusion ring of rank four is non-simple by Theorem \ref{thm:Burnside}. 
On the other hand, a pointed fusion ring of rank four is isomorphic to $\Z[G]$ for some group $G$ of order four. 
It therefore cannot be simple since $G$ contains a subgroup of order two. 

\section{The non-self-dual setting} \label{sec:nonselfdual}

Recall that throughout the classification, the multiplication is commutative and
satisfies the unit, duality, and Frobenius reciprocity axioms. 
Also, the structure constants are nonnegative integers and $\det(N_X)=0$. Thus, $\mathcal{R}$ is a fusion ring if and only if the associativity axiom holds, or, equivalently, the multiplication matrices of $X, Y$, and $Z$ pairwise commute. 

Suppose in this section that $\mathcal{R}$ is not self-dual and we determine explicit necessary and sufficient conditions on the structure constants that turn $\mathcal{R}$ into a fusion ring.

\begin{lemma}\label{lem:singular-dual}
If $\mathcal R$ is a fusion ring, then $2^*=2$.
\end{lemma}
\begin{proof}
Otherwise, after interchanging $Y$ and $Z$ if needed, we have $Y=X^*$
and $Z=Z^*$. Commutativity, duality, and Frobenius reciprocity axioms give
\[
 N_X=\begin{pmatrix}
 0&1&0&0\\
 0&N_{222}&N_{223}&N_{224}\\
 1&N_{222}&N_{222}&N_{234}\\
 0&N_{242}&N_{224}&N_{244}
 \end{pmatrix},\qquad N_Y=N_X^T, \; \; \text{and} \; \; N_Z=\begin{pmatrix} 
0 & 0 & 0 & 1 \\ 
0 & N_{234} & N_{224}  &N_{244} \\ 
0 & N_{224} & N_{234} & N_{244} \\ 
1 & N_{244} & N_{244} & N_{444}
 \end{pmatrix}. 
\]
Given a real column vector $\mathbf{v}$, we set as usual $\lVert \mathbf{v}\rVert=\sqrt{\mathbf{v}^T\mathbf{v}}$ to be its Euclidean norm. 
Since $N_X$ commutes with $N_Y$, and so with its transpose, we get
\[ \lVert N_X\mathbf{v} \rVert^2= (N_X\mathbf{v} )^T N_X\mathbf{v} =\mathbf{v}^TN_X N_X^T\mathbf{v}= \lVert N_X^T \mathbf{v}\rVert^2\] for every real column vector $\mathbf{v}$.
Their kernels therefore coincide. 
Since $N_X$ has determinant zero, this kernel then contains a nonzero vector
$\mathbf{v}=(v_1,v_2,v_3,v_4)^T$. 
The first rows of $N_X$ and $N_Y$ give $v_2=v_3=0$, and the remaining
equations are $N_{224}v_4=0, \; v_1+N_{234}v_4=0$, and $N_{244}v_4=0$.
It then follows that $v_4\ne0$ and $N_{224}=N_{244}=0$.

The relation $N_XN_Z=N_ZN_X$ leads to $1+N_{444}N_{234}=N_{234}^2$ and $1+N_{444}N_{242}=N_{242}^2$. 
Hence, $N_{234}(N_{234}-N_{444})=1$, that is, $N_{234}=1$ and $N_{444}=0$. It follows that $N_{242}=1$. 

The relation $N_XN_Y=N_YN_X$ gives $N_{223}^2=2+N_{222}^2$, that is, $(N_{223}-N_{222})(N_{223}+N_{222})=2$.
Thus, $N_{223}+N_{222}=2$ and $N_{223}-N_{222}=1$, leading to $N_{223}=3/2$, a contradiction.
\end{proof}

\begin{lemma}\label{lem:nonself}
Suppose that $2^*=2$. Then $\mathcal R$ is a fusion ring if and only if
\[
 N_X=\begin{pmatrix}
 0&1&0&0\\1&N_{222}&1&1\\0&1&0&0\\0&1&0&0
 \end{pmatrix},\quad
 N_Y=\begin{pmatrix}
 0&0&1&0\\0&1&0&0\\0&0&0&1\\1&0&0&0
 \end{pmatrix},\; \; \text{and} \; \; 
 N_Z=\begin{pmatrix}
 0&0&0&1\\0&1&0&0\\1&0&0&0\\0&0&1&0
 \end{pmatrix},
\]
where $N_{222}\in\mathbb Z_{\geq0}$ is arbitrary. Such a ring is never simple and it is integral if and only if $N_{222}=2$. In that case, the basic Frobenius--Perron dimensions of $X,Y$, and $Z$ are $3,1$, and $1$, respectively.
\end{lemma}
\begin{proof}
For the converse, the displayed matrices easily satisfy the axioms of a fusion ring. Also, $Y^2=Z$, $YZ=\mathbf{1}$, and $Z^2=Y$, so
$\{\mathbf{1},Y,Z\}$ spans a proper nontrivial fusion subring of $\mathcal{R}$.

We now treat the forward direction. Since $2^*=2$, we must have $X^*=X, Y^*=Z,$ and $Z^*=Y$. 
The commutativity, duality, and Frobenius reciprocity axioms give \[N_X= \begin{pmatrix} 
0 & 1 & 0 & 0 \\ 
1 & N_{222} & N_{223} &N_{223} \\ 
0 & N_{223} & N_{233} & N_{234} \\ 
0 & N_{223} & N_{234} & N_{244}
 \end{pmatrix}, \qquad N_Y=\begin{pmatrix} 
0 & 0 & 1 & 0 \\ 
0 & N_{223} &N_{233} & N_{234} \\ 
0 & N_{234} & N_{333} & N_{334} \\ 
1 & N_{233} &N_{333} & N_{333} 
 \end{pmatrix}, \; \; \text{and} \; \;  N_Z=N_Y^T.\] 
 Comparing the $(4,1)$-entries of $N_XN_Y=N_YN_X$ gives $N_{244}=N_{233}$. 
The determinant condition now yields
$0=\det(N_X)=N_{234}^2-N_{233}^2$, hence $N_{233}=N_{234}$. 
Comparing the different entries of the equality $N_XN_Y=N_YN_X$ leads to $N_{233}N_{333}=N_{233}N_{334}$ and \[N_{223}^2+2N_{233}^2=1+N_{222}N_{233}+2N_{333}N_{223}=N_{222}N_{233}+N_{223}N_{334}+N_{223}N_{333}.\]
We get $N_{223}(N_{334}-N_{333})=1$, that is, $N_{223}=1$ and $N_{334}=1+N_{333}$. 
Since $N_{233}N_{333}=N_{233}N_{334}$, we conclude $N_{233}=0$; whence $N_{333}=0$ and $N_{334}=1$. 
The equivalence is now established.

We have $\FPdim(Y)=\FPdim(Z)=1$. 
The eigenvalues of $N_X$, which are algebraic integers, are
\[ 0,\quad0,\quad \frac{N_{222}+\sqrt{N_{222}^2+12}}2,\quad \frac{N_{222}-\sqrt{N_{222}^2+12}}2.\]
The positive eigenvalue is integral exactly when $N_{222}^2+12=l^2$ for
some positive integer $l$. The factorization
$(l-N_{222})(l+N_{222})=12$, together with a straightforward case analysis, shows that the unique solution in positive integers is $(l,N_{222})=(4,2)$. Note that $\FPdim(X)=3$.
\end{proof}

\section{The self-dual setting: the degenerate case} \label{sec:selfdualdeg}

From now on, we assume that $\mathcal{R}$ is self-dual and we will use the matrices defined in \eqref{eq:matrices} throughout the text.

Recall that $\mathcal{R}$ is a fusion ring if and only if the multiplication matrices of $X, Y$, and $Z$ commute pairwise. 
We also recall that $\det(N_X)=0$. Expanding the determinant of $N_X$ along the first row, then the resulting $3\times 3$ determinant along the first column, gives $0=\det(N_X)=e^2-df$. 

The goal of this section consists in determining when $\mathcal{R}$ is a fusion ring in the degenerate case $df=0$, which forces $e=0$. 
Swapping $Y$ and $Z$ sends the tuple $(a,b,c,d,e,f,g,h,i,j)$ to $(a,c,b,f,e,d,j,i,h,g)$. 
As we are interested in isomorphism classes of vanishing fusion rings, we may assume that $(d,h)$ is less than or equal to $(f,i)$ with respect to the usual lexicographic order.  
In particular, it forces $d=0$. 

\begin{lemma}\label{lem:fusionringselfdeg}
The free module $\mathcal R$ is a fusion ring if and only if $c^2+f^2=2+af+cj$ and 
\[N_X= \begin{pmatrix} 
0 & 1 & 0 & 0 \\ 
1 & a & 1 & c \\ 
0 & 1 & 0 & 0 \\ 
0 & c & 0 & f
 \end{pmatrix}, \; \; N_Y=\begin{pmatrix} 
0 & 0 & 1 & 0 \\ 
0 & 1& 0 &0 \\ 
1 & 0 & 0 & 0 \\ 
0 & 0 &0 & 1
 \end{pmatrix}, \; \; \text{and} \; \;  N_Z=\begin{pmatrix} 
0 & 0 & 0 & 1 \\ 
0 & c &0 & f \\ 
0 & 0 & 0 & 1 \\ 
1 & f & 1 & j
 \end{pmatrix}. \]  
 In other words, $\mathcal{R}$ is a fusion ring if and only if $(a,b,c,d,e,f,g,h,i,j)=(a,1,c,0,0,f,0,0,1,j)$ and $c^2+f^2=2+af+cj$.  
 In particular, $\mathcal{R}$ cannot be a simple fusion ring. 
\end{lemma}

\begin{proof}
Since $e=0$, the multiplication matrices commute if and only if 
\[ b^2=1+bg+ch,\; \; bc=bh+ci,\; \; hf=0, \; \; c^2+f^2=1+af+bi+cj,\; \; bf=if, \; \; h^2+i^2=1+gi+hj\] and the backward direction of the first part of the lemma follows. 

We now assume that $\mathcal{R}$ is a fusion ring and we claim that $h=0$ and $b=i$. 
This is clear if $f>0$ since $hf=0$ and $bf=if$. 
If $f=0$, the first and fourth equations above give $b,c>0$ and the first also gives
$\gcd(b,c)=1$. The second then implies that $c$ divides $h$ and $b$ divides $i$.
Write $h=cr$ and $i=bs$ with $r,s\in\mathbb Z_{\geq0}$, and note that the second equation yields $r+s=1$.
The case $(r,s)=(1,0)$ being excluded since $h\leq i$, we actually have $(r,s)=(0,1)$, that is, $h=0$ and $i=b$. 

Thus, $b(b-g)=1$; whence $b=1$ and $g=0$. All remaining displayed equations above reduce to
$c^2+f^2=2+af+cj$, proving the equivalence.
Finally, $\mathcal{R}$ cannot be simple since $Y^2=\mathbf{1}$. 
\end{proof}

Assume that $\mathcal{R}$ is a fusion ring. 
The matrix $N_Z$ has determinant zero since its first and third rows coincide. 
Since we are interested in isomorphism classes of fusion rings, we can therefore swap $X$ and $Z$ if needed, that is, interchange the tuples $(a,b,c,d,e,f,g,h,i,j)=(a,1,c,0,0,f,0,0,1,j)$ and $(j,i,f,h,e,c,g,d,b,a)=(j,1,f,0,0,c,0,0,1,a)$.
Lemma \ref{lem:fusionringselfdeg} states that $(c,f)\neq (0,0)$ (otherwise, $XZ$ would be zero), and so we can reduce our study to the case $f>0$. 

We characterize below the integral case of Lemma \ref{lem:fusionringselfdeg} up to isomorphism. 

\begin{lemma}\label{lem:one-zero}
Assume that $\mathcal{R}$ is a fusion ring and that $f>0$. 
Then $\mathcal{R}$ is integral precisely when $(a,c,f,j)=(0,1,1,0), (1,0,2,1)$ or $(1,0,2,5)$. 
In these three cases, the basic Frobenius-Perron dimensions in the basis $\{\mathbf{1},X,Y,Z\}$ are then $(1,2,1,2), (1,2,1,3)$, and $(1,2,1,6)$, respectively. 
\end{lemma}
\begin{proof}
It is easy to check that $\mathcal{R}$ is integral if $(a,c,f,j)$ is one of the three prescribed values, and that its basic Frobenius--Perron dimensions coincide with what was claimed. We now assume that $\mathcal{R}$ is integral and we show below that $(a,c,f,j)$ is one of these three prescribed values. 
For brevity, put $x=\FPdim(X)$ and $z=\FPdim(Z)$.
Note that $\FPdim(Y)=1$. 

The characteristic polynomials of $N_X$ and $N_Z$ are \[ T(T^3-(a+f)T^2+(af-c^2-2)T+2f) \; \; \mathrm{and} \; \; T(T^3-(j+c)T^2+(cj-f^2-2)T+2c),   \] respectively. 
Hence, $x$ must divide $2f$ and $2c$ is a multiple of $z$.
Write $2f=mx$ and $2c=nz$, where $m$ and $n$ are nonnegative integers. 
Note that $m$ is positive since $f>0$.
Next, the equations \[X^2=\mathbf{1}+aX+Y+cZ, \; \; XZ=cX+fZ, \; \; \text{and} \; \; Z^2=\mathbf{1}+fX+Y+jZ \] lead to 
\begin{equation} \label{eq: lemd=0f>0}
 x^2=2+ax+cz, \; \; \; \; xz=cx+fz, \; \;  \text{and} \; \; z^2=2+fx +jz. 
\end{equation}
Multiplying the second equation by $2$ implies $n+m=2$.

Suppose first that $(m,n)=(2,0)$. 
Then $f=x, c=0$, and the first equation in \eqref{eq: lemd=0f>0} gives $x(x-a)=2$; whence $x=f=2$ and $a=1$.
Plugging these constants in the third equation of \eqref{eq: lemd=0f>0} now leads to $z^2=jz+6$. 
Since $z$ is an integer, it means that $j^2+24=l^2$ for some integer $l$.
Again, the factorization $(l-j)(l+j)=24$, together with an easy case analysis, shows that either $j=1$ or $j=5$, which ends this case. 

Suppose now that $(m,n)=(1,1)$. 
Then $x=2f, z=2c$, and the first and third equations in \eqref{eq: lemd=0f>0} provide $1+af+c^2=2f^2$ as well as $1+f^2+jc=2c^2$. 
In particular, $\gcd(c,f)=1$ and
\[
 c^2+f^2+1=f(3f-a)=c(3c-j).
\]
Consequently, $c$ divides $3f-a$, and so $cf$ divides $c^2+f^2+1$.  We obviously have $c^2-2cf+f^2+1 >0$ and the nonnegativity of $a$ and $j$
gives
\[
 2<\frac{c^2+f^2+1}{cf}
   =\frac{3f-a}{c}=\frac{3c-j}{f}
   \leq\min\left\{\frac{3f}{c},\frac{3c}{f}\right\}\leq3. 
\]
Since $(3f-a)/c$ is an integer by the foregoing, we must have $a=3(f-c)$ and $j=3(c-f)$. The nonnegativity of $a$ and $j$ leads to $c=f$. 
This implies that $a=j=0$, and so $c=f=1$.
\end{proof}

\begin{rem} \label{rem:isod=0f>0}
For dimension reasons, the three integral fusion rings of Lemma \ref{lem:one-zero} are not isomorphic. 
In conclusion, there is one isomorphism class of integral vanishing fusion ring in the non-self-dual case (Lemma \ref{lem:nonself}) and there are three in the self-dual degenerate case (Lemma \ref{lem:one-zero}).
\end{rem}

\section{The self-dual setting: the nondegenerate case} \label{sec:selfdualnondegenerate}

We now study the nondegenerate case $df>0$ in this section and the next. 
Recall that the integer $qc-pb$ is nonnegative.
This section aims to prove Theorem \ref{thm:main} and to derive from it a characterization of the integral case. 
From now on, we use the constants defined in \eqref{eq:kappapq}. 

The structure of this section is as follows: the first subsection states a few intermediate lemmas allowing us to reduce the proof of Theorem \ref{thm:main} to establishing the equivalence between $(1)$ and $(2c)$. 
This equivalence constitutes the computational part of the proof. 

We show in the second subsection that the pairwise commutativity of the matrices $N_X, N_Y$, and $N_Z$ leads to strong arithmetic conditions, which turn out to be equivalent to two systems of two linear matrix equations. 
Solutions of the first and second systems are given by the matrices \[ \begin{pmatrix}
g & h  \\ 
h & i
 \end{pmatrix} \; \; \text{and} \; \; \begin{pmatrix} 
h & i  \\ 
i & j
 \end{pmatrix}, \] respectively. 
 Next, a linear algebra computation proves that these two systems have a unique solution, namely $cR-pK$ for the first and $-bR+qK$ for the second. 
 The theorem follows. 
 
The third and last subsection shows that the integral case of Theorem \ref{thm:main} precisely matches $\kappa=\lambda^2$, where $\lambda=\gcd(d,f)$ until the end of this text. 

\subsection{Reductions} \label{subsec:reductions}
Write $d=\lambda x$ and $f=\lambda y$.
The equality $e^2=df$ implies that $xy=(e/\lambda)^2$ is a perfect square.
The integers $x$ and $y$ being coprime, every prime exponent in $x$ or in $y$ is therefore even. 
Hence, $x=q^2$ and $y=p^2$ for some coprime positive integers $p$ and $q$. 
From now on, and until the end of this text, we will extensively use without further explanation the fact that $d=\lambda q^2, e=\lambda pq$, and $f=\lambda p^2$.   

\begin{lmm} \label{lmm:2a}
Assume that $\mathcal{R}$ is a fusion ring. Then $qc-pb=1$. 
\end{lmm}

\begin{proof}
The leading $3\times3$ minor of $N_X$ has determinant $-d\ne0$.
Since $\det(N_X)=0$, the kernel of $N_X$ is one-dimensional and it is easy to check that it is spanned by $\mathbf{v}=(qc-pb, 0, p, -q)^T$.

Since $\mathcal{R}$ is a fusion ring, we know that the matrices $N_X, N_Y$, and $N_Z$ commute pairwise. 
Consequently, $N_Y\mathbf{v}$ and $N_Z\mathbf{v}$ lie in the kernel of $N_X$ as well, and are therefore collinear to $\mathbf{v}$.
The first coordinate of the vector $N_Y\mathbf{v}$ being $p>0$, we get $qc-pb\neq 0$ and a direct calculation yields \[ N_Y\mathbf{v}= \frac{p}{qc-pb}\mathbf{v} \; \; \; \; \text{and} \; \; \; \; N_Z\mathbf{v}= \frac{-q}{qc-pb}\mathbf{v}.  \]
The entries of matrices $N_Y$ and $N_Z$ being integers, it follows that their eigenvalues are algebraic integers. This forces $qc-pb > 0$ to divide both $p$ and $q$, and their coprimality gives $qc-pb=1$.
\end{proof}

We now show that the conditions $(1)$ and $(3)$ are equivalent. 

\begin{lmm} \label{lmm:(1)<=>(2)}
If $\mathcal{R}$ is a fusion ring, then it is necessarily simple. 
\end{lmm}

\begin{proof}
We have $(b,c)\neq (0, 0)$ according to Lemma \ref{lmm:2a}. 
Thus, $X^2=\mathbf{1}+aX+bY+cZ$ contains either $Y$ or $Z$. 
The fact that the integers $d$ and $f$  are positive shows that $Y^2=\mathbf{1}+dX+gY+hZ$ and $Z^2=\mathbf{1}+fX+iY+jZ$ contain $X$. 
Finally, the positivity of $e$ proves that  $XY=bX+dY+eZ$ contains $Z$ and $XZ=cX+eY+fZ$ contains $Y$. 
It follows that $\mathcal{R}$ is simple.
\end{proof}

Recall that $\mathcal{R}$ is a fusion ring if and only if the matrices $N_X, N_Y$, and $N_Z$ commute pairwise. 
This criterion tells us that we need to perform six matrix multiplications, which is too cumbersome due to the complexity of calculations. The next lemma explains how to reduce this number to four. 

\begin{lmm} \label{lmm:4computations}
Assume that $N_X$ commutes with both $N_Y$ and $N_Z$. Then $N_Y$ and $N_Z$ commute. 
\end{lmm}

\begin{proof}
For brevity, put $A=N_YN_Z-N_ZN_Y$ and let us show that it is the zero matrix. 
One has \[A(1,0,0,0)^T=N_Y(0,0,0,1)^T-N_Z(0,0,1,0)^T=(0,e,h,i)^T-(0,e,h,i)^T=(0,0,0,0)^T.\]
Consequently, the first column of $A$ is the zero vector. 
By observing that both $N_Y$ and $N_Z$ are symmetric, we infer that $A^T=-A$, that is, $A$ is skew-symmetric. 
The entries located in the first row or on the main diagonal of $A$ are therefore $0$. Hence, there are integers $\alpha,\beta$, and $\gamma$ for which \[A=\begin{pmatrix} 
0 & 0 & 0 & 0 \\ 
0 & 0 & \alpha & \beta \\ 
0 & -\alpha & 0 & \gamma \\ 
0 & -\beta & -\gamma & 0
 \end{pmatrix}.\] 
 
Next, $A$ commutes with $N_X$ since the latter commutes with both $N_Y$ and $N_Z$. The first row of $N_XA$ being $(0, 0, \alpha, \beta)$, while that of $AN_X$ is $(0,0,0,0)$, we deduce that $\alpha=\beta=0$.
Finally, the $(4,4)$-entry of $N_XA$ is $e\gamma$, while that of $AN_X$ is $-e\gamma$; whence $\gamma=0$ since $e>0$. 
\end{proof}

Thanks to Lemma \ref{lmm:(1)<=>(2)}, Theorem \ref{thm:main} would follow if $(1)$ and $(2)$ are equivalent. 
By Lemma \ref{lmm:2a}, if $qc-pb\neq 1$, then neither $(1)$ nor $(2)$ can hold; they are therefore equivalent. 
We now suppose that $qc-pb= 1$.  
Recall that $(2a)$ and $(2c)$ imply $(2b)$ by Remark \ref{rem:divisibility}. 
Thus, $(1)$ and $(2)$ are equivalent if and only if $(1)$ and $(2c)$ are. 
The proof of the latter is the purpose of the next subsection.
 
 Recall that Lagrange identity yields 
 \begin{equation} \label{cor:Legendre}
 (q^2+p^2)(b^2+c^2)-(bq+cp)^2=(qc-pb)^2=1.
 \end{equation}
 
\subsection{The linear matrix systems} \label{subsec:linearmatrix}
Given a pair $\mathbf{v}=(v_1,v_2)$ of complex numbers, we set \[\Theta(\mathbf{v}) = \begin{pmatrix}
v_1g+v_2h & v_1h+v_2i  \\ 
v_1 h +v_2 i & v_1 i+v_2 j
 \end{pmatrix}.\]  

 \begin{lmm} \label{lmm:eqcommute}
 The matrix $N_X$ commutes with both $N_Y$ and $N_Z$ if and only if  \[ 
 \Delta(\mathbf{v}):=\Theta(\mathbf{v})\begin{pmatrix}
b  \\ 
c 
 \end{pmatrix} - (bv_1+cv_2)\begin{pmatrix}
b  \\ 
c 
 \end{pmatrix}-(b^2+c^2-\kappa)(qv_1+pv_2)\begin{pmatrix}
q  \\ 
p 
 \end{pmatrix} +\mathbf{v}^T  \] and  \[\Omega(\mathbf{v}):= \begin{pmatrix}
q^2 & pq  \\ 
pq & p^2
 \end{pmatrix} \Theta(\mathbf{v}) - \Theta(\mathbf{v}) \begin{pmatrix}
q^2 & pq  \\ 
pq & p^2
 \end{pmatrix} + (qv_1+pv_2)\begin{pmatrix}
0 & -1  \\ 
1 & 0
 \end{pmatrix}\]
 are zero matrices for all pairs of complex numbers $\mathbf{v}=(v_1,v_2)$.
 \end{lmm}

\begin{proof}
Clearly, $N_X$ commutes with both $N_Y$ and $N_Z$ if and only if $A=N_XM(\mathbf{v})-M(\mathbf{v})N_X$ is the zero matrix for all pairs of complex numbers $\mathbf{v}=(v_1, v_2)$, where \[M(\mathbf{v}):=v_1N_Y+v_2N_Z = \begin{pmatrix} 
0 & 0 & v_1 & v_2 \\ 
0 & v_1b+v_2c & v_1d+v_2e & v_1e+v_2f \\ 
v_1 &v_1d+v_2e & v_1g+v_2h & v_1h+v_2i\\ 
v_2 & v_1e+v_2f & v_1h +v_2i & v_1i+v_2j
 \end{pmatrix}. \] 
 Note that the lemma would immediately follow once the equality \[A= \begin{pmatrix} 
0 & 0 & 0 \\ 
0 & 0 & \Delta(\mathbf{v})^T \\ 
0 & - \Delta(\mathbf{v}) &\lambda \Omega(\mathbf{v})
 \end{pmatrix} \] is proved.
 Because $N_X$ and $M(\mathbf{v})$ are symmetric matrices, it follows that $A$ is skew-symmetric, and so all entries on the main diagonal are $0$.   
 Next, 
  \[ A (1,0,0,0)^T= N_X (0,0,v_1,v_2)^T - M(\mathbf{v})(0,1,0,0)^T=(0,0,0,0)^T,\] which proves that the first column of $A$, and therefore its first row, contains only $0$. 

Recall that $qc-pb=1$. 
We now show that the block of $A$ given by the entries $(3,3), (3,4), (4,3)$, and $(4,4)$ is equal to $\lambda \Omega(\mathbf{v})$.
Since these two matrices are skew-symmetric, it suffices to prove that the $(4,3)$-entry of $A$ matches the $(2,1)$-entry of $\lambda \Omega(\mathbf{v})$, namely \[\lambda  \left(pq(v_1g+v_2h)+p^2(v_1h+v_2i)-q^2(v_1h+v_2i)-pq(v_1i+v_2j)+(qv_1+pv_2)(qc-pb)\right).\]
We infer the desired equality since a direct computation shows that the $(4,3)$-entry of $A$ equals \[e(v_1g+v_2h)+f(v_1h+v_2i)-d(v_1h+v_2i)-e(v_1i+v_2j)+c(v_1d+v_2e)-b(v_1e+v_2f).\]
 
We now show that the block of $A$ given by the entries $(3,2)$ and $(4,2)$ is equal to $-\Delta(\mathbf{v})$. 
An easy calculation establishes that this block is equal to \[ -\mathbf{v}^T+(v_1b+v_2c)\begin{pmatrix} 
b \\ 
c 
\end{pmatrix} - \Theta(\mathbf{v}) \begin{pmatrix} 
b \\ 
c 
\end{pmatrix} + \begin{pmatrix} 
(d^2+e^2-ad)v_1+e(d+f-a)v_2 \\ 

\\

e(d+f-a)v_1+(e^2+f^2-af)v_2
\end{pmatrix}.  \] 
Using the relations $e^2=df$ and $(d+f-a)\lambda =b^2+c^2-\kappa$ easily yields the desired formula. 
\end{proof}

  This lemma shows that it is relevant to study the system of linear matrix equations \[ (S_\mathbf{v}) \; \; \begin{cases} 
M\begin{pmatrix}
b  \\ 
c 
 \end{pmatrix} = (bv_1+cv_2)\begin{pmatrix}
b  \\ 
c 
 \end{pmatrix}+(b^2+c^2-\kappa)(qv_1+pv_2)\begin{pmatrix}
q  \\ 
p 
 \end{pmatrix} -\mathbf{v}^T \\ 
 
 \\

 \begin{pmatrix}
q^2 & pq  \\ 
pq & p^2
 \end{pmatrix} M- M\begin{pmatrix}
q^2 & pq  \\ 
pq & p^2
 \end{pmatrix} = (qv_1+pv_2) \begin{pmatrix}
0 & 1  \\ 
-1 & 0
 \end{pmatrix} 
 \end{cases}, \] with $\mathbf{v}=(v_1,v_2)$ a pair of complex numbers. 
 Clearly, the set of solutions is either empty or an affine space. 
  For the convenience of the reader, we recall that $qc-pb=1$ and that \[ \tau =\frac{c^2+b^2+(p^2+q^2)\kappa}{bq+pc}-(bq+pc), \qquad R =\begin{pmatrix}
   2qb-q^2\tau&qc+bp-pq\tau \\
   pb+cq-pq\tau&2cp-p^2\tau
   \end{pmatrix},\] and  \[ K=\begin{pmatrix}
   b^2-1+q^2(c^2+b^2-\kappa)&bc+pq(c^2+b^2-\kappa)\\
   bc+pq(c^2+b^2-\kappa)&c^2-1+p^2(c^2+b^2-\kappa)
   \end{pmatrix}.\]

   \begin{lmm} \label{lmm:solvesystem}
  Let $\mathbf{v}=(v_1,v_2)$ be a pair of complex numbers.
  Then the system $(S_\mathbf{v})$ has a unique solution given by $M=v_1(cR-pK)+v_2(-bR+qK)$.
  \end{lmm}
   
   \begin{proof}
  For uniqueness, it suffices to claim that the zero matrix is the unique matrix satisfying \[ 
M\begin{pmatrix}
b  \\ 
c 
 \end{pmatrix} = \begin{pmatrix}
0  \\ 
0 
 \end{pmatrix}  \; \; \text{and} \; \;  \begin{pmatrix}
q^2 & pq  \\ 
pq & p^2
 \end{pmatrix} M=M\begin{pmatrix}
q^2 & pq  \\ 
pq & p^2
 \end{pmatrix}. \]
Multiplying the second equation by $(b,c)^T$ shows that $(qb+pc)M (q, p)^T=(0,0)^T$. 
The uniqueness follows since $qb+pc>0$ and the matrix of $M$ in the basis $\{ (b,c)^T, (q, p)^T\}$ is the zero matrix. 

For existence, the system $(S_\mathbf{v})$ and the proposed solution are linear in $\mathbf{v}$.
It then suffices to show the lemma for $\mathbf{v}=(b,c)$ and $\mathbf{v}=(q,p)$, for which the
expected solutions are $K$ and $R$, respectively.

Let us start with the case $\mathbf{v}=(b,c)$. 
 An easy computation ensures that $K$ satisfies the first equation. 
 For the second one, we use the symmetry of $K$ to infer that the matrix \[ \begin{pmatrix}
q^2 & pq  \\ 
pq & p^2
 \end{pmatrix} K- K\begin{pmatrix}
q^2 & pq  \\ 
pq & p^2
 \end{pmatrix} \] is skew-symmetric. 
 It is then enough to prove that its $(1,2)$-entry is $qb+pc$, which follows directly.
 
We now deal with the case $\mathbf{v}=(q, p)$. 
  A direct calculation provides \[R\begin{pmatrix}
b  \\ 
c 
 \end{pmatrix}= \begin{pmatrix}
b(qb+cp)+q\left( (bq+pc)^2-(p^2+q^2)\kappa \right) \\ 

\\

c(bq+cp)+p\left((bq+pc)^2-(p^2+q^2)\kappa \right) 
 \end{pmatrix},\]
 and \eqref{cor:Legendre} claims that $R$ satisfies the first equation. 
 Again, thanks to the symmetry of $R$, the latter would satisfy the second equation if the $(1,2)$-entry of the matrix \[\begin{pmatrix}
q^2 & pq  \\ 
pq & p^2
 \end{pmatrix} R- R\begin{pmatrix}
q^2 & pq  \\ 
pq & p^2
 \end{pmatrix}\] is $q^2+p^2$, which can be checked directly. 
   \end{proof}
   
  We now deduce Theorem \ref{thm:main}. 
 Recall that $\mathcal{R}$ is a fusion ring if and only if $N_X$ commutes with both $N_Y$ and $N_Z$. 
Since the system $(S_\mathbf{v})$ is linear in $\mathbf{v}$, Lemma \ref{lmm:eqcommute} tells us that $\mathcal{R}$ is a fusion ring if and only if $\Theta(1,0)$ and $\Theta(0,1)$ are solutions of $(S_{(1,0)})$ and $(S_{(0,1)})$, respectively. Finally, the definition of $\Theta(\mathbf{v})$, together with Lemma \ref{lmm:solvesystem}, ensures that $\mathcal{R}$ is a fusion ring if and only if 
\begin{flalign*}
 && \begin{pmatrix}
g & h  \\ 
h & i
 \end{pmatrix}= cR-pK \; \; \text{and} \; \; \begin{pmatrix} 
h & i  \\ 
i & j
 \end{pmatrix}= -bR+qK. && \square
 \end{flalign*}
   
\subsection{The integral specialization} \label{ss-section:integral case} 
Theorem \ref{thm:main} classifies all self-dual vanishing fusion rings of rank four with $df>0$, but it does not tell us anything about the integral case. The current subsection will show that this case is characterized by the simple condition $\kappa=\lambda ^2$. 

We suppose in this subsection that $\mathcal{R}$ is a fusion ring and recall that $e^2=df>0$. 
A direct computation shows that $\chi(T)$, the characteristic polynomial of $N_X$, is \[ \chi(T)= T\left(T^3-(a+d+f)T^2+(ad+af-b^2-c^2-1)T+dc^2+fb^2-2bce+d+f\right). \] 
By observing that $fb^2-2ebc+dc^2=\lambda (pb-qc)^2=\lambda $, we simply have 
\begin{equation} \label{eq:polychar}
 \chi(T)= T\left(T^3-(a+d+f)T^2+(ad+af-b^2-c^2-1)T+d+f+\lambda \right). 
 \end{equation}

\begin{lmm} \label{lmm:kappaequFPdim}
We have $\kappa=\lambda ^2$ if and only if $\FPdim(X)=d+f+\lambda $. In addition, if one of these equivalent conditions is satisfied, then \[\chi(T)=T(T-d-f-\lambda )(T^2-(a-\lambda )T-1). \] 
\end{lmm}

\textit{Proof.}
First, it follows from \eqref{eq:polychar} that $d+f+\lambda $ is a root of $\chi(T)$ if and only if $d+f+\lambda$ is a root of the polynomial $T^2-(a+d+f)T+ad+af-b^2-c^2$ if and only if $\kappa=\lambda^2$.
In particular, this shows the converse. We now suppose that $\kappa=\lambda^2$. 

If $d+f+\lambda, \beta$, and $\gamma$ are the roots of $\chi(T)/T$, then $(d+f+\lambda)+\beta+\gamma= a+d+f$ and $(d+f+\lambda)\beta\gamma=-(d+f+\lambda )$. Equivalently, $\beta+\gamma=a-\lambda $ and $\beta\gamma=-1$, which proves that \[\chi(T)=T(T-d-f-\lambda )(T^2-(a-\lambda )T-1). \] 

The roots of $T^2-(a-\lambda )T-1$, namely $\beta$ and $\gamma$, are real since the discriminant is positive and have opposite signs.
By Frobenius--Perron theory, $\FPdim(X)$ is the largest positive real root of $\chi$. 
Consequently, $d+f+\lambda\geq \max\{\beta, \gamma\}$ if and only if $T^2-(a-\lambda )T-1$ is a nonnegative number when evaluating at $T=d+f+\lambda$.
This can be done by noticing that 
\begin{flalign*}
&& (d+f+\lambda) \left(d+f+\lambda -(a-\lambda)\right) \geq 3 \left( \frac{b^2+c^2-\kappa}{\lambda }+2\lambda\right) =3 \left(\frac{b^2+c^2}{\lambda}+\lambda\right) \geq 3. && \square
\end{flalign*} 

\begin{prop} \label{prop:firstcharinteger}
The fusion ring $\mathcal{R}$ is integral if and only if $\kappa=\lambda ^2$. 
Moreover, if one of these equivalent properties holds, then \[\left(\FPdim(X), \FPdim(Y), \FPdim(Z)\right)=\left(\lambda(1+p^2+q^2), \; b+q(bq+pc), \; c+p(bq+pc)\right).\]
\end{prop}

\begin{proof}
To ease notation, put $x=\FPdim(X), y=\FPdim(Y)$, and $z=\FPdim(Z)$. 
The dimension equations of relations $XY=bX+dY+eZ$ and $XZ=cX+eY+fZ$ yield  \[ (S) \; \; \begin{cases} (x-d)y-ez=bx \\ 
 -ey + (x-f)z=cx 
 \end{cases}. \]

If $\kappa=\lambda ^2$, then $x=d+f+\lambda=\lambda(1+p^2+q^2)$ by Lemma \ref{lmm:kappaequFPdim} and we obtain a linear system of two equations in two unknowns. 
The determinant of its associated matrix is \[(f+\lambda )(d+\lambda )-e^2= \lambda ^2+(f+d)\lambda =x\lambda .\] We therefore have \[ \begin{pmatrix} 
y \\ 
z 
\end{pmatrix}= \frac{1}{\lambda } \begin{pmatrix} d+\lambda  & e \\ 
e & f+\lambda 
\end{pmatrix}  \begin{pmatrix} 
b \\ 
c 
\end{pmatrix} = \begin{pmatrix} q^2+1& pq \\ 
pq & p^2+1
\end{pmatrix}  \begin{pmatrix} 
b \\ 
c 
\end{pmatrix}\] and $\mathcal{R}$ is integral. Furthermore, the values of $y$ and $z$ correspond to those claimed in the statement. 

We now assume that $\mathcal{R}$ is integral so that $x$ is an integer. Multiplying the first equation of $(S)$ by $q$ and the second by $p$ yield $qxy-qdy-qez-epy+pxz-pfz=x(bq+pc)$. The equalities $qe=pd$ and $ep=fq$ lead to $(x-d-f)(qy+pz)=x(bq+pc)>0$; whence $x>d+f\geq (d+f+\lambda )/2$. 
On the other hand, $x$ must divide the constant term of $\chi(T)/T$, namely $d+f+\lambda $. This forces $x=d+f+\lambda $ and we derive $\kappa=\lambda ^2$ thanks to Lemma \ref{lmm:kappaequFPdim}. 
\end{proof}

We now finish this section by proving a series of corollaries. 

\begin{cor} \label{cor:basicdividesglobal}
If $\mathcal{R}$ is integral, then $\FPdim(X)$ divides $\FPdim(\mathcal{R})$ if and only if $\lambda=1$. 
\end{cor}

\begin{proof}
Since $\mathcal{R}$ is integral, Proposition \ref{prop:firstcharinteger} ensures that $\kappa=\lambda^2$ and that 
\begin{align*}
\FPdim(\mathcal{R}) & =1+\lambda^2(1+p^2+q^2)^2+(b+q(bq+pc))^2+(c+p(bq+pc))^2 \\
& = 1+(bq+pc)^2+b^2+c^2 +\lambda^2(1+p^2+q^2)^2+(1+p^2+q^2)(bq+pc)^2.
\end{align*}
It now arises from \eqref{cor:Legendre} that \[ \FPdim(\mathcal{R}) = (1+p^2+q^2)(b^2+c^2+\lambda^2(1+p^2+q^2)+(bq+pc)^2).\]

Since $\lambda$ divides $\kappa$, and so $b^2+c^2$ by \eqref{eq:kappapq}, a new use of \eqref{cor:Legendre} then shows that the second factor in the last displayed equation is congruent to $-1$ modulo $\lambda$. 
Since $\FPdim(X)=\lambda(1+p^2+q^2)$ by Proposition \ref{prop:firstcharinteger}, we infer that the latter divides $\FPdim(\mathcal{R})$ if and only if $\lambda=1$. 
\end{proof}

\begin{cor} \label{cor:2integers2quad}
If $\mathcal{R}$ is integral, then $\chi$ has two integer roots and two quadratic roots. 
\end{cor}

\textit{Proof.}
Lemma \ref{lmm:kappaequFPdim} and Proposition \ref{prop:firstcharinteger} reduce the proof to showing that $T^2-(a-\lambda)T-1$ is an irreducible polynomial over $\Q$. If not, then it would have two integer roots whose product is $-1$, that is, $-1$ and $1$, and so $a=\lambda$. 
But then, $\lambda^2=\kappa=b^2+c^2+(\lambda-f-d)\lambda$, which implies $b^2+c^2=(f+d)\lambda=\lambda^2(p^2+q^2)$. 
Equation \eqref{cor:Legendre} now derives a contradiction, since otherwise
\begin{flalign*}
 && 1= (q^2+p^2)(b^2+c^2)-(bq+cp)^2= \left(\lambda(p^2+q^2)-(bq+cp) \right)\left( \lambda(p^2+q^2)+(bq+cp)\right). && \square
 \end{flalign*}

\begin{cor} \label{cor:unique-vanishing}
If $\mathcal{R}$ is integral, then both $N_Y$ and $N_Z$ have nonzero determinant. 
\end{cor}

\begin{proof}
We only show that $N_Y$ has a nonzero determinant, the second case arising from the first one by swapping $Y$ and $Z$. 
Assume by contradiction that $N_Y$ has determinant zero. 
Expanding the determinant of $N_Y$ along the first row, then the resulting $3\times 3$ determinant along the first column, gives $0=\det(N_Y)=e^2-bi$. 
In particular, $b$ is non-zero since $e$ is a positive integer. 

The leading $3\times3$ minor of $N_Y$ has determinant $-b\ne0$, and so the kernel of $N_Y$ is a one-dimensional vector space. 
Since the entries of $N_Y$ are integers, this kernel is spanned by a nonzero rational element $\mathbf{v}=(v_1,v_2,v_3,v_4)^T$. 
Clearly, \[v_3=0, \; \; bv_2+ev_4=0, \; \; v_1+dv_2+hv_4=0, \; \; \text{and} \; \; ev_2+iv_4=0.\]
We cannot have $v_2=0$, since otherwise we would obtain $v_1=v_4=0$. 

Since $\mathcal{R}$ is a fusion ring, the matrices $N_X$ and $N_Y$ commute. 
In particular, $N_X\mathbf{v}$ lies in the kernel of $N_Y$, and it is therefore collinear to $\mathbf{v}$. 
This coefficient of proportionality must be rational since $N_X$ and $\mathbf{v}$ have rational entries. 
The coefficient in the first row of $N_X\mathbf{v}$ being $v_2\neq 0$, we must have  $N_X\mathbf{v}= \FPdim(X)\mathbf{v}$ since Corollary \ref{cor:2integers2quad} claims that $\FPdim(X)$ is the unique nonzero rational eigenvalue of $N_X$. 
In fact, it also tells us that the eigenvalues of $N_X$ are pairwise distinct, which implies that the eigenspace of $N_X$ associated to the eigenvalue $\FPdim(X)$ is one-dimensional and spanned by $(1, \FPdim(X), \FPdim(Y), \FPdim(Z))$.
We get a contradiction since $v_3=0$.
\end{proof}

\section{The integral classification and the special case \texorpdfstring{$\lambda=1$}{lambda=1}}\label{sec:fibonacci}
In this section, we propose to prove Theorems \ref{thm:classintegral} and \ref{thm:classintegral1}.
Recall that $df>0$, that is, $pq>0$.

\subsection{Two Diophantine equations} \label{ss-sec:Markov}
Example \ref{ex:family} shows that it is easy to construct fusion rings from Theorem \ref{thm:main}. 
However, none of them is integral since $\kappa \equiv 2 \pmod{3}$ while $\lambda^2\not\equiv 2 \pmod{3}$.

In fact, even with the criterion $\kappa=\lambda^2$ at hand, it is quite challenging to find by hand a case where $\mathcal{R}$ is integral. 
This rigidity stems from Markov equations, which admit highly constrained positive integer solutions \cite{Mar}. 
More precisely, we will see that $(\lambda(p^2+q^2), bq+pc)$ is a solution of the Markov equation $x^2+y^2+1=3xy$. 
The special case $\lambda=1$ will also show that $(b,q)$ and $(p,c)$ are solutions of the equation $x^2+y^2-3xy=\pm 1$. 
Recall that $(F_n)_{n\in \Z}$ is the Fibonacci sequence.

\begin{thm} \label{thm:solveMarkov}
The nonnegative integer solutions $(k,x,y)$ with $x\geq y$ of the equation $x^2+y^2+1=kxy$ are $(3,1,1)$ and $(3, F_{2n+1}, F_{2n-1})$, where $n$ is any positive integer. 
\end{thm}

\begin{proof}
A result of Hurwitz with $n=3$ gives $k\leq 3$ \cite[Satz 6]{Hur}. Since this diophantine equation clearly has no real solution if $k\leq 2$, we must have $k=3$. We conclude thanks to \cite{BK}. 
\end{proof}

\begin{thm} \label{thm:solveMarkov2}
The nonnegative integer solutions $(x,y)$ with $x\geq y$ of the equation $x^2+y^2-1=3xy$ are $(F_{2n}, F_{2n-2})$, where $n$ is any positive integer. 
\end{thm}

\begin{proof}
See \cite{BK}. 
\end{proof}

\subsection{Proof of Theorem \ref{thm:classintegral}} \label{ss-sec:proofintegralcase}
In case $(b,c)\neq (0,0)$, write $t$ for the quotient of $b^2+c^2+(p^2+q^2)\kappa$ by $bq+pc$ (note that $bq+pc$ is a positive integer since $p$ and $q$ are positive).
\begin{lemma}\label{lem:scale-one}
Assume that $qc-pb=1$ and $\kappa=\lambda^2$. Then  \[ \left( \lambda(p^2+q^2) \right)^2 +(bq+pc)^2+1 = \frac{t}{\lambda} \left(\lambda(p^2+q^2)\right) (bq+pc). \] 
Moreover, if $t=3\lambda$, then $g,h,i$, and $j$ are as in Theorem \ref{thm:main} $(2c)$ if and only if
\[\begin{cases} 
g= q^2\left(p\lambda^2-3\lambda c+b\right)+p+2b+pb^2 \\

\\

h = q\left(pb+c^2-3\lambda pc+p^2\lambda^2\right) \\

\\

i= p \left( qc+c^2-3\lambda pc+ p^2 \lambda^2 \right) \\ 

\\ 

j= (p^2-qc+2)c+ \lambda p^2(3b- \lambda q)-q
\end{cases}. \]
\end{lemma}

\begin{proof}
The displayed equation is a direct consequence of \eqref{cor:Legendre}. 
Since $\tau=3\lambda-bq-pc$, one has
\begin{align*}
g &= 2bqc-q^2c(3\lambda-bq-pc)-pb^2+p-pq^2(c^2+b^2-\lambda^2) \\ 
& = q^2(p\lambda^2-3\lambda c)+p+2b(1+pb)+q^2b(1+pb)-pb^2-pq^2b^2 \\ 
& = q^2\left(p\lambda^2-3c\lambda+b\right)+p+2b+pb^2
\end{align*}
and the remaining three calculations rely on the same approach. 
\end{proof}

In case $\mathcal{R}$ is a simple fusion ring, Theorem \ref{thm:main} tells us that $t$ is an integer, but we did not say anything more about its value. 
The rigidity of Markov equations actually forces $t$ to be $3\lambda$ if $\mathcal{R}$ is integral.
This explains why we can provide a more explicit classification in the integral case (and also why this is challenging to find one by hand). 

\begin{lemma}\label{lem:scale-onebis}
If $\mathcal{R}$ is a simple integral fusion ring, then $t=3\lambda, b$ and $c$ are positive, and there is an integer $n\geq 3$ for which $\lambda(p^2+q^2)=F_{2n+1}$ and $qb+pc=F_{2n-1}$.  
\end{lemma}

\begin{proof}
Proposition \ref{prop:firstcharinteger} yields $\kappa=\lambda^2$ and \eqref{eq:kappapq} proves that $\lambda$ divides $b^2+c^2$. 
Equation \eqref{cor:Legendre} now claims that $\lambda$ is coprime to $bq+pc$. 
We infer that $\lambda(bq+pc)$ divides $b^2+c^2+(p^2+q^2)\kappa$, that is, $\lambda$ divides $t$. 
 Lemma \ref{lem:scale-one} now tells us that  $(\lambda(p^2+q^2), bq+pc)$ is a solution of the Markov equation $x^2+y^2+1=(t/\lambda)xy$ and Theorem \ref{thm:solveMarkov} gives $t=3\lambda$. 

We now show the positivity of $b$ and $c$. Since $qc-pb=1$ by Theorem \ref{thm:main} $(2a)$, we cannot have $c=0$. 
If $b=0$, then $q=c=1$. 
Since $\lambda$ divides $b^2+c^2$ by the foregoing, $\lambda$ must be $1$. 
Theorem \ref{thm:main} now shows that $h$ equals the $(1,1)$-entry of $K$, namely $-1$, a contradiction. 

The equality $t=3\lambda$, that is, $b^2+c^2+(p^2+q^2)\lambda^2=3\lambda(bq+pc)$, and \eqref{eq:kappapq} give \[a = \frac{\lambda^2-(b^2+c^2)}{\lambda}+\lambda(p^2+q^2) = \lambda+ 2\lambda (p^2+q^2)-3(bq+pc) < 3\lambda (p^2+q^2)-3(bq+pc).\] The nonnegativity of $a$ yields $\lambda(p^2+q^2) > bq+pc$ and Theorem \ref{thm:solveMarkov} ensures that $\lambda(p^2+q^2)=F_{2n+1}$ and $bq+pc=F_{2n-1}$ for some positive integer $n$.  

If $n=1$, then $bq+pc=1$, contradicting the positivity of $b,c,p$, and $q$. 
If $n=2$, then $bq+pc=2$ and we must have $b=q=p=c=1$, contradicting $qc-pb= 1$. 
Thus $n\geq 3$ and we are done. 
\end{proof}

We now prove Theorem \ref{thm:classintegral}. 
Assume that $\mathcal{R}$ is a simple integral fusion ring. 
The expression for $\lambda$ comes from Lemma \ref{lem:scale-onebis}. 
Then, Proposition \ref{prop:firstcharinteger} gives $\kappa=\lambda^2$ and \eqref{eq:kappapq} leads to the value of $a$. 
We obtain those of $b$ and $c$ by solving the equations $qc-pb=1$ and $qb+pc=F_{2n-1}$, the last one arising from Lemma \ref{lem:scale-onebis}. The expressions for $d,e$, and $f$ are obvious, and those for $g,h,i$, and $j$ follow from Lemma \ref{lem:scale-one} since $t=3\lambda$ according to Lemma \ref{lem:scale-onebis}. 

Conversely, assume that the constants are as in the displayed system and let us show that $\mathcal{R}$ is a simple integral fusion ring. 
Comparing the value of $a$ with that in \eqref{eq:kappapq} yields $\kappa=\lambda^2$. 

 We have $qc-pb=1$, and so Theorem \ref{thm:main} $(2a)$. 
Clearly, $qb+pc=F_{2n-1}$ and $\lambda(p^2+q^2)=F_{2n+1}$. 
Combining Theorem \ref{thm:solveMarkov} and Lemma \ref{lem:scale-one} gives $t=3\lambda$ and Theorem \ref{thm:main} $(2b)$ holds. 
Lemma \ref{lem:scale-one} gives Theorem \ref{thm:main} $(2c)$. 
Theorem \ref{thm:main} and Proposition \ref{prop:firstcharinteger} now prove the desired conclusion. \qed

\subsection{The special case \texorpdfstring{$\lambda=1$}{lambda=1}} \label{ss-sec:casegcd1}
As explained right after Theorem \ref{thm:classintegral}, the constants defined in this theorem can be either negative or nonintegral. 
Remarkably, if $\kappa=\lambda=1$, then for all integers $n\geq 3$, there are parameters $p$ and $q$ making all of them both nonnegative and integral. 

For the convenience of the reader, we record some well-known properties of the Fibonacci sequence that we will use several times in this subsection. 
The first one is a trivial calculation. 

\begin{lmm} \label{lmm:Fiboelementary}
We have $F_n=3F_{n-2}-F_{n-4}$ for all integers $n$. 
\end{lmm}

\begin{proof}
Clear since $F_n=F_{n-1}+F_{n-2}=2F_{n-2}+F_{n-3}=3F_{n-2}-F_{n-4}$. 
\end{proof}

 The next one is Vajda's identity, a proof of which can be found in  \cite[equation (20a)]{Vaj}.

\begin{thm} \label{thm:Vajda}
We have $F_{u+v}F_{u+w}-F_uF_{u+v+w}=(-1)^u F_vF_w$ for all integers $u,v$, and $w$. 
\end{thm}

The next formula derives from \cite[equation (8)]{Vaj} with $m=n+1$ and $G_n=F_{n+u}$. 

\begin{thm} \label{thm:Vajda2}
We have $F_nF_{n+u}+F_{n+1}F_{n+1+u}= F_{2n+u+1}$ for all integers $n$ and $u$. 
\end{thm}

In particular, $F_n^2+F_{n+1}^2=F_{2n+1}$ for all integers $n$. 
Brown roughly proved the converse of this identity in \cite{Bro}, in answer to Problem B-$62$ posed by Brother Alfred. 
The condition $n,i,j\geq 2$ is imposed in order to rule out the case $F_1=F_2=1$. 

\begin{thm} \label{thm:Brown}
If $F_{2n+1}= F_i^2+F_j^2$ for some integers $n,i,j\geq 2$, then $(i,j)\in \{(n,n+1), (n+1,n)\}$. 
\end{thm}

Basically, the result below states that the additional constraint $\lambda=1$ in Theorem \ref{thm:classintegral} is enough to force $b,c,p$, and $q$ to be Fibonacci numbers. 

\begin{lmm} \label{lmm:tracedet}
Assume that $\kappa=1$ and that there exists an integer $n\geq 2$ such that $p^2+q^2=F_{2n+1}$. 
Denote by $l$ and $k$ the odd and even integers in $\{n,n+1\}$, respectively. 
Then $q>b, qc-pb=1$, and $t=3$ if and only if $(p,q,c,b)=(F_l, F_k, F_{l-2}, F_{k-2})$. 
\end{lmm}

\begin{proof}
Conversely, it is clear that $q>b$. We have $l-k=\pm 1$, and so $F_{l-k}=1$. 
Thus, Theorem \ref{thm:Vajda} applied to $u=k-2, v=2$, and $w=l-k$ gives $qc-pb=F_kF_{l-2}-F_lF_{k-2}=(-1)^k=1$ since $k$ is even. 
On the one hand, Theorem \ref{thm:Vajda2} and Lemma \ref{lmm:Fiboelementary} provide \[ b^2+c^2+p^2+q^2=F_{n-2}^2+F_{n-1}^2+F_n^2+F_{n+1}^2=F_{2n-3}+F_{2n+1}=3F_{2n-1}.\] 
On the other hand, Theorem \ref{thm:Vajda2} with $u=-2$ yields $3(bq+pc)=3(F_nF_{n-2}+F_{n+1}F_{n-1})=3F_{2n-1}$. As $\kappa=1$ by assumption, we get $t=3$ and the converse follows. 

We now show the forward direction. 
Let us consider the matrix
\[
 A=\begin{pmatrix}q&b\\p&c\end{pmatrix}
   \begin{pmatrix}0&-1\\1&3\end{pmatrix}
   \begin{pmatrix} c & -b\\-p&q\end{pmatrix}=\begin{pmatrix} pq+bc-3pb & -q^2-b^2+3bq \\ p^2+c^2-3pc & -pq-bc+3qc\end{pmatrix}.
\]
It has trace $3$, determinant $1$, and is symmetric since $\kappa=1$ and $t=3$. 
If $\alpha= pq+bc-3pb$ and $\beta= p^2+c^2-3pc$, then $\alpha(3-\alpha)=1+\beta^2>0$; whence $\beta=\pm 1$. 
Note that $q^2+b^2-3bq=-\beta=\mp 1$. 

We have $p>c$. 
Indeed, if $c>p$, then $1=qc-pb\geq (b+1)(p+1)-pb=1+b+p$, leading to $b=p=0$, a contradiction. 
If $c=p$, then we must have $c=1$.  
Thus, $\beta=-1, q-b=1$, and $q^2+b^2-3bq=1$, implying $b=0$ and $q=1$. 
Hence, $F_{2n+1}=2$, that is $n=1$, a contradiction. 

 Theorems \ref{thm:solveMarkov} and \ref{thm:solveMarkov2} now provide integers $r$ and $s$ such that $(p,q,c,b)=(F_r, F_s, F_{r- 2}, F_{s-2})$. 
Theorem \ref{thm:Brown} establishes that $(r,s)\in \{(n,n+1), (n+1,n)\}$.
In particular, $r-s=\pm 1$. 
Finally, Theorem \ref{thm:Vajda} with $u=s-2, v=2$, and $w=r-s$ yields $1=F_sF_{r-2}-F_{s-2}F_r=(-1)^s$. 
This forces $s$ to be the even integer belonging to $\{n,n+1\}$, namely $k$.
We derive $r=l$.
\end{proof}

\begin{lmm} \label{lmm:lmmnonneg}
Suppose that $\lambda=1$.
Let $n\geq 3$ be an integer, and let $l$ and $k$ be the odd and even integers in $\{n,n+1\}$, respectively.  
If $(p,q,c,b)=(F_l, F_k, F_{l-2}, F_{k-2})$, then $a, g, h, i$, and $j$ are as in Theorem \ref{thm:classintegral} if and only if \[ \begin{cases} 
a=1+F_{2n+1}-F_{2n-3} \\ 

\\ 

g= F_k^2(F_{k-2}-F_{l-4})+F_l+2F_{k-2}+F_lF_{k-2}^2 \\ 

\\ 

h=F_k(F_{k-2}F_l-1), \; \; \; \; \; \; i=F_l(F_kF_{l-2}-1) \\ 

\\ 

j= F_{l-2}(F_l^2-F_{l-2}F_k+2)+F_l^2F_{k-4}-F_k
\end{cases}. \] 
In addition, all these integers are nonnegative. 
\end{lmm}

\textit{Proof.}
Theorem \ref{thm:Vajda2} yields $p^2+q^2=F_n^2+F_{n+1}^2=F_{2n+1}$ and $b^2+c^2=F_{n-2}^2+F_{n-1}^2=F_{2n-3}$. 
This gives the expression for $a$. 
Those for $g$ and $j$ come from Lemma \ref{lmm:Fiboelementary} with $n=l$ and $n=k$, respectively. 
As $l$ is odd, Theorem \ref{thm:solveMarkov} gives $c^2-3pc+p^2=-1$. 
The expressions for $h$ and $i$ follow.

Except for $g$ and $j$, it is clear that all constants are nonnegative.
If $l=3$, which only occurs when $n=3$, and so $k=4$, we have $g=6$ and $j=0$. 
Assume that $l\geq 5$. 
The nonnegativity of $g$ comes from the inequality $F_{k-2}\geq F_{l-4}$ since $ l-k= \pm 1$. 
If $k=4$, then $l=5$ and $n=4$, leading to $j=39$. 
If $k>4$, the nonnegativity of $j$ arises from $F_l^2F_{k-4}-F_k\geq F_l^2-F_k> 0$ and 
\begin{flalign*}
 && F_l^2-F_{l-2}F_k\geq F_l^2-F_{l-2}F_{l+1}=F_l^2-(F_l-F_{l-1})(F_l+F_{l-1})=F_{l-1}^2. && \square
 \end{flalign*}

We now prove Theorem \ref{thm:classintegral1}. 
We have $\lambda=1$.
If $\mathcal{R}$ is a simple integral fusion ring, then Theorem \ref{thm:classintegral} gives $p^2+q^2=F_{2n+1}, b < q$, and $qc-pb=1$. 
We get $t=3$ by Lemma \ref{lem:scale-onebis} and $\kappa=1$ by Proposition \ref{prop:firstcharinteger}. Lemmas \ref{lmm:tracedet} and \ref{lmm:lmmnonneg} yield the desired values of $a,\dots,j$, which are all nonnegative.

Theorem \ref{thm:Vajda2} with $u=-2$ gives $bq+pc=F_{k-2}F_k+F_lF_{l-2}=F_{n-2}F_n+F_{n+1}F_{n-1}=F_{2n-1}$.
The formulas regarding the basic Frobenius--Perron dimensions now follow from Proposition \ref{prop:firstcharinteger}.

Conversely, let us show that the ten structure constants are as in the displayed system of Theorem \ref{thm:classintegral}.
The values of $d$ and $f$ give $(q,p)=(F_k,F_l)$.
Theorem \ref{thm:Vajda2} with $u=0$ tells us that $p^2+q^2=F_n^2+F_{n+1}^2=F_{2n+1}$. 
Next, Theorem \ref{thm:Vajda} with $u=k, v=-2$, and $w=2n+1-k$ provides \[b(p^2+q^2)-qF_{2n-1}= F_{k-2}F_{2n+1}-F_kF_{2n-1}=(-1)^k F_{-2}F_{2n+1-k}=-F_l=-p\] since $k$ is even and $F_l=F_{2n+1-k}$.  
Similarly, Theorem \ref{thm:Vajda} with $u=l, v=-2$, and $w=2n+1-l$ gives $c(p^2+q^2)-pF_{2n-1}=q$. 
The expressions for $d,e$, and $f$ are trivial. 
For the remaining constants $g,h,i$, and $j$, this is obvious by Lemma \ref{lmm:lmmnonneg}. 
\qed

\section{Two applications of our study}\label{subsec:frobenius}

Recall that an integral fusion ring is $1$-Frobenius if all its basic Frobenius--Perron dimensions divide
its global Frobenius--Perron dimension.  We also recall that a fusion ring with unit $\mathbf{1}$ is said to be Burnside if $N_X$ has determinant $0$ for any basis element $X$ satisfying $XX^*\neq \mathbf{1}$. 

We conclude this paper by showing that no simple integral fusion ring of rank four is $1$-Frobenius and that a vanishing fusion ring of rank four is Burnside if and only if it is non-simple. 

We start with the $1$-Frobenius case, for which we provide a brief proof: the starting point is that every simple integral fusion ring of rank four is self-dual (Lemma \ref{lmm:selfdual}). 
If such a ring were $1$-Frobenius, then Proposition \ref{prop:determinant-general}, stated in greater generality because it might be of independent interest, would force the multiplication matrix of every nonunit basis element to have determinant zero, conflicting with Corollary \ref{cor:unique-vanishing}. 
 
 \begin{lemma} \label{lmm:selfdual}
Every simple integral fusion ring of rank four is self-dual.
\end{lemma}

\begin{proof}
Suppose otherwise. Up to isomorphism, the distinguished basis is
$\{\mathbf{1},x,x^*,y\}$, with $x\ne x^*$ and $y=y^*$.
Put $a=\operatorname{FPdim}(x)=\operatorname{FPdim}(x^*)$ and $b=\operatorname{FPdim}(y)$. 
Note that both $a$ and $b$ are positive integers since the fusion ring is integral. 

As explained at the beginning of the proof of Lemma \ref{lem:singular-dual}, we have \[ N_x=\begin{pmatrix}
 0&1&0&0\\
 0& \alpha &\beta & \gamma \\
 1& \alpha & \alpha & \delta \\
 0& \delta & \gamma & \varepsilon
 \end{pmatrix}, \; \; N_{x^*}=N_x^T, \; \; \text{and} \; \; N_y=\begin{pmatrix} 
0 & 0 & 0 & 1 \\ 
0 & \delta & \gamma & \varepsilon \\ 
0 & \gamma & \delta & \varepsilon \\ 
1 & \varepsilon & \varepsilon & \zeta
 \end{pmatrix},  
\] where all entries are nonnegative integers. 
We have $\varepsilon>0$, since otherwise $\{1,y\}$ would span a proper fusion subring, contradicting the simplicity assumption. 

The dimension equations in particular give $ab=(\delta+\gamma)a+\varepsilon b$ and $b^2= 1+2\varepsilon a+\zeta b$. The latter implies that $a$ and $b$ are coprime, and the former therefore claims that $a$ divides $\varepsilon$. 
The relation $(a-\varepsilon)b=(\delta+\gamma)a\geq 0$ leads to $a\geq \varepsilon$, and so $a=\varepsilon$. 
As a result, we must have $\gamma=\delta=0$. 

Consequently, $x^2=\alpha x+\beta x^*, (x^*)^2=\beta x+ \alpha x^*,$ and $xx^*=1+\alpha(x+x^*)$, meaning that $\{1, x, x^*\}$ spans a proper fusion subring, a contradiction. 
\end{proof}
 
\begin{proposition} \label{prop:determinant-general}
Let $\mathcal F$ be a simple integral fusion ring of rank $r\geq3$ with basis $\{\mathbf{1}, X_1,\dots, X_{r-1}\}$. 
Let $X_i$ be a self-dual basis element. 
Put $x=\FPdim(X_i)$ and $D=\FPdim(\mathcal F)$. Then $x\gcd(x,D)$ divides $\det(N_{X_i})$ and  $\vert\det(N_{X_i})\vert<x^{r-2}$.
Moreover, if $r=4$ and $x$ divides $D$, then $\mathrm{det}(N_{X_i})=0$.
\end{proposition}
\begin{proof}
By relabelling the basis elements if needed, we can assume that $i=1$.  
It is clear that $N_{X_1}$ is both symmetric and nonnegative. 
Since $X_1$ is self-dual and $\mathcal{F}$ is simple, the fusion ring generated by $X_1$ is the whole fusion ring $\mathcal{F}$, that is, the graph of $N_{X_1}$ is connected. 
It follows from \cite[Theorem~1.7.15]{May} that $N_{X_1}$ is irreducible, and so the eigenvalue $x$ is simple \cite[Theorem 1.9.6]{May}.

By definition, $x$ is a root of the characteristic polynomial $\chi$ of $N_{X_1}$; whence $\chi(T)=(T-x)Q(T)$ for some polynomial $Q\in \Z[T]$ satisfying $\chi'(x)=Q(x)$. 
 Also, the column vector \[
 \mathbf{v}=(1, \FPdim(X_1),\ldots,\FPdim(X_{r-1}))^T \] satisfies $N_{X_1}\mathbf{v}=x\mathbf{v}$ and $\mathbf{v}^T\mathbf{v}=D$. 
 Let $I$ be the identity matrix of size $r$. 
 
 Since $N_{X_1}$ is a real symmetric matrix, Thompson-McEnteggert formula \cite{TM, CZ} tells us that the adjugate of the matrix $xI-N_{X_1}$ is $(Q(x)/D) \mathbf{v}\mathbf{v}^T$.
The adjugate of an integer matrix is an integer matrix.
The $(1,1)$-entry of the matrix $\mathbf{v}\mathbf{v}^T$ being $1$, we deduce that $D$ divides $Q(x)$.
Since $x$ also divides $Q(x)-Q(0)$, we infer that $\gcd(x,D)$ divides $Q(0)$. 
The first claim of the statement now follows from the equalities $(-1)^r\det(N_{X_1})=\det(-N_{X_1})=-xQ(0)$.

There is a row vector $w$ with $r-2$ entries and a square matrix $C$ of size $r-2$ such that
\[
 N_{X_1}=\begin{pmatrix}
 0&1&0\\ 1&N_{111}&w\\ 0&w^T&C
 \end{pmatrix}.\] Note that $\det(N_{X_1})=-\det C$. 
 Let $B$ be the matrix obtained from $N_{X_1}$ by replacing all entries outside the block $C$ with $0$.
Since $N_{X_1}$ is irreducible, \cite[Theorem 1.9.7]{May} states that the spectral radius $\rho$ of $B$, which equals that of $C$, is less than $x$. 
The determinant of a square matrix being the product of its eigenvalues counted with multiplicity, we get 
$\vert \det(N_{X_1})\vert= |\det C|\leq\rho^{r-2}<x^{r-2}$.
The second assertion follows since a nonzero multiple of $x^2$
has absolute value at least $x^2$.
\end{proof}

\begin{rem}
The proof uses simplicity only to make $N_{X_1}$ irreducible.
The proposition therefore holds for any integral fusion ring with
an irreducible multiplication matrix at a self-dual nonunit basis
element. 
\end{rem}

We now prove Theorem \ref{thm:frobenius-obstruction}.
First, $\mathcal{F}$ is self-dual by Lemma \ref{lmm:selfdual}. 
Assume that there is a nonunit basis element $X$ such that $\FPdim(X)$ divides $\FPdim(\mathcal{F})$. 
Then $\det(N_X)=0$ according to Proposition \ref{prop:determinant-general}, and so $\mathcal{F}$ is a simple integral self-dual vanishing fusion ring. 
Our study then shows that we are in the setting of Subsection \ref{ss-section:integral case}. 
The first part of the theorem now arises from Corollary \ref{cor:unique-vanishing} and the second from Corollary \ref{cor:basicdividesglobal} and Theorem \ref{thm:classintegral1}. \qed \\

We now prove Theorem \ref{thm:Burnside}. 
Conversely, up to isomorphism, the non-simple vanishing fusion rings of rank four are given by Lemmas \ref{lem:nonself} and \ref{lem:fusionringselfdeg}. 
Those in Lemma \ref{lem:nonself} are Burnside since $N_X$ has determinant zero, while $N_Y$ and $N_Z$ are permutation matrices. 
Those in Lemma \ref{lem:fusionringselfdeg} are also Burnside since $N_X$ and $N_Z$ have determinant zero, while $N_Y$ is a permutation matrix.

Suppose that $\mathcal{R}$ is a nonpointed Burnside fusion ring of rank four. 
It is therefore vanishing and we assume by contradiction that it is simple. 
Our study then claims that $\mathcal{R}$ is self-dual and $df>0$. 
Recall that $0=\det(N_X)=e^2-df$ and that $b$ and $c$ are coprime by Theorem \ref{thm:main} $(2a)$. 

The matrices \eqref{eq:matrices} show that $Y^2$ and $Z^2$ are different from $\mathbf{1}$.  
By definition of the Burnside property, we must have $\det(N_Y)=\det(N_Z)=0$. 
Swapping $X$ and $Y$, respectively $Z$, sends $(a,b,c,d,e,f,g,h,i,j)$ to $(g,d,h,b,e,i,a,c,f,j)$, respectively $(j,i,f,h,e,c,g,d,b,a)$.
The same argument, with $X$ replaced by $Y$ or $Z$, yields $e^2=df=bi=ch$ and $d$ and $h$, as well as $i$ and $f$, are coprime.  
It then follows that $b$ divides $h$, which divides $f$, which divides $b$; whence $b=h=f$. 
Similarly, $c$ divides $i$, which divides $d$, which divides $c$; whence $c=i=d$. 

Since $b$ and $c$ are coprime, we have $p=\sqrt{b}$ and $q=\sqrt{c}$. 
It arises from Theorem \ref{thm:main} $(2a)$ that $1=qc-pb=q^3-p^3=(q-p)(q^2+qp+p^2)$, a contradiction since $p$ and $q$ are positive integers. 
 \qed

\subsection*{Acknowledgments}
The authors thank Dmitry Rybin for sharing an exchange with GPT-Pro via the CUHK-Shenzhen AI Math Problems database.

S\'ebastien Palcoux is supported by the National Natural Science
Foundation of China (NSFC), Grant no.~12471031.
Arnaud Plessis is supported by the Beijing Municipal Natural Science
Foundation (No.~IS25018).
GPT-6 Astra was used as a research assistant. The authors take full
responsibility for the final manuscript.

\end{document}